\documentclass{amsart}
\usepackage{amssymb,amsmath,color,amsthm}
\usepackage{color}
\usepackage{tikz}
\usepackage{tkz-graph}
\usetikzlibrary{calc, positioning}
\usepackage{pgf}
\usepackage{bm}
\usepackage[center]{caption}
\usepackage{pgfplots}
 \usepackage[left=1in,right=1in,top=0.9in,bottom=0.9in]{geometry}

\def \hom {\operatorname{hom}}

\newcommand{\tind}{t_{\textup{ind}}}

\newcommand{\ww}{\mathbf{w}}
\newcommand{\RR}{\mathbb{R}}
\newcommand{\ZZ}{\mathbb{Z}}
\newcommand{\NN}{\mathbb{N}}
\newcommand{\bN}{\mathbf{N}}

\newtheorem{theorem}{Theorem}[section]

\newtheorem{problem}[theorem]{Problem}
\newtheorem*{theorem*}{Theorem}
\newtheorem*{problem*}{Problem}

\newtheorem{lemma}[theorem]{Lemma}
\newtheorem{corollary}[theorem]{Corollary}
\newtheorem{proposition}[theorem]{Proposition}
\theoremstyle{definition}
\newtheorem{definition}[theorem]{Definition}

\newtheorem{claim}[theorem]{Claim}

\theoremstyle{remark}

\newcommand{\upfour}[0]{
\begin{tikzpicture}[baseline={([yshift=-.5ex]current bounding box.center)}]
\node[fill=black, circle, inner sep=1pt, minimum size=0.1cm] (1) {};
\node[fill=black, circle, inner sep=1pt, minimum size=0.1cm] (2) [above = 0.3cm of 1] {};
\node[fill=black, circle, inner sep=1pt, minimum size=0.1cm] (3) [right = 0.3cm of 2] {};
\node[fill=black, circle, inner sep=1pt, minimum size=0.1cm] (4) [below = 0.3cm of 3] {};
\node[fill=black, circle, inner sep=1pt, minimum size=0.1cm] (5) [right = 0.3cm of 4] {};
\draw (1)--(2)--(3)--(4)--(5);
\end{tikzpicture}}

\newcommand{\uHtwo}[0]{
\begin{tikzpicture}[baseline={([yshift=-.5ex]current bounding box.center)}]
\node[fill=black, circle, inner sep=1pt, minimum size=0.1cm] (1) {};
\node[fill=black, circle, inner sep=1pt, minimum size=0.1cm] (2) [below left = 0.3cm of 1] {};
\node[fill=black, circle, inner sep=1pt, minimum size=0.1cm] (3) [below right =0.3cm of 1] {};
\draw (3)--(1)--(2);
\end{tikzpicture}}

\newcommand{\utriangleedge}[0]{
\begin{tikzpicture}[baseline={([yshift=-.5ex]current bounding box.center)}]
\node[fill=black, circle, inner sep=1pt, minimum size=0.1cm] (1) {};
\node[fill=black, circle, inner sep=1pt, minimum size=0.1cm] (2) [below left = 0.3cm of 1] {};
\node[fill=black, circle, inner sep=1pt, minimum size=0.1cm] (3) [below right = 0.3cm of 1] {};
\node[fill=black, circle, inner sep=1pt, minimum size=0.1cm] (4) [above = 0.3cm of 1] {};
\draw (1)--(3)--(2)--(1)--(4);
\end{tikzpicture}}

\newcommand{\uvedge}[0]{
\begin{tikzpicture}[baseline={([yshift=-.5ex]current bounding box.center)}]
\node[fill=black, circle, inner sep=1pt, minimum size=0.1cm] (1) {};
\node[fill=black, circle, inner sep=1pt, minimum size=0.1cm] (2) [below = 0.3cm of 1] {};
\node[draw=none, fill=none, circle, inner sep=1pt, minimum size=0.1cm] (4) [left = 0.05cm of 1] {};
\draw (2)--(1);
\end{tikzpicture}}

\newcommand{\uKfour}[0]{
\begin{tikzpicture}[baseline={([yshift=-.5ex]current bounding box.center)}]
\node[fill=black, circle, inner sep=1pt, minimum size=0.1cm] (1) {};
\node[fill=black, circle, inner sep=1pt, minimum size=0.1cm] (2) [right = 0.3cm of 1] {};
\node[fill=black, circle, inner sep=1pt, minimum size=0.1cm] (3) [below = 0.3cm of 2] {};
\node[fill=black, circle, inner sep=1pt, minimum size=0.1cm] (4) [below = 0.3cm of 1] {};
\draw (4)--(2)--(1)--(3)--(4)--(1)--(2)--(3);
\end{tikzpicture}}

\newcommand{\ucfour}[0]{
\begin{tikzpicture}[baseline={([yshift=-.5ex]current bounding box.center)}]
\node[fill=black, circle, inner sep=1pt, minimum size=0.1cm] (1) {};
\node[fill=black, circle, inner sep=1pt, minimum size=0.1cm] (2) [above = 0.3cm of 1] {};
\node[fill=black, circle, inner sep=1pt, minimum size=0.1cm] (3) [right = 0.3cm of 2] {};
\node[fill=black, circle, inner sep=1pt, minimum size=0.1cm] (4) [below = 0.3cm of 3] {};
\draw (1)--(2)--(3)--(4)--(1);
\end{tikzpicture}}

\title{Undecidability of polynomial inequalities in weighted graph homomorphism densities}
\thanks{Grigoriy Blekherman was partially supported by NSF grant DMS-1901950. Annie Raymond was partially supported by NSF grant DMS-2054404. Fan Wei was partially supported by NSF grant DMS-1953958.}

\author{Grigoriy Blekherman}
\address{School of Mathematics, Georgia Institute of Technology,
686 Cherry Street,
Atlanta, GA 30332}\email{greg@math.gatech.edu}

\author{Annie Raymond}
\address{Department of Mathematics and Statistics,
Lederle Graduate Research Tower, 1623D,
University of Massachusetts Amherst,
710 N. Pleasant Street,
Amherst, MA 01003} \email{raymond@math.umass.edu}

\author{Fan Wei}
\address{Department of Mathematics,
Fine Hall, Washington Road,
Princeton, NJ 08544}\email{fanw@princeton.edu}

\begin{document}
\maketitle

\begin{abstract}
Many problems and conjectures in extremal combinatorics concern polynomial inequalities between homomorphism densities of graphs where we allow edges to have real weights. Using the theory of graph limits, we can equivalently evaluate polynomial expressions in homomorphism densities on \emph{kernels} $W$, i.e., symmetric, bounded, and measurable functions $W$ from $[0,1]^2 \to \mathbb{R}$.
In 2011, Hatami and Norin proved a fundamental result that it is undecidable to determine the validity of polynomial inequalities in homomorphism densities for graphons (i.e., the case where the range of $W$ is $[0,1]$, which corresponds to unweighted graphs, or equivalently, to graphs with edge weights between $0$ and $1$). The corresponding problem for more general sets of kernels, e.g., for all kernels or for kernels with range $[-1,1]$, remains open.  For any $a > 0$, we show undecidability of polynomial inequalities for any set of kernels which contains all kernels with range $\{0,a\}$.
This result also answers a question raised by Lov\'asz about finding computationally effective certificates for the validity of homomorphism density inequalities in kernels.
\end{abstract}

\section{Introduction}

A graph $G$ has vertex set $V(G)$ and edge set $E(G)$. All graphs are assumed to be simple, without loops or multiple edges. Given two simple graphs $G$ and $H$, let $\hom(H, G)$  denote the number of {\it homomorphisms} from $H$ to $G$, which is the set of maps from $V(H)$ to $V(G)$ that send edges of $H$ to edges of $G$. Furthermore, let $t(H,G):=\frac{\hom(H,G)}{|V(G)|^{|V(H)|}}$ denote the \emph{homomorphism density of $H$ in $G$}, i.e., the probability that a random map from $V(H)$ to $V(G)$ is a homomorphism.  One can extend the usual definition of graph homomorphisms to include target graphs $G$ with edge weights $\ww: E(G) \to \mathbb{R}$, denoted as $G_\ww$:
$$\hom(H,G_{\ww}) := \sum_{\substack{\varphi: V(H)\rightarrow V(G_{\ww}):\\ \varphi \textup{ is a homomorphism}}} \prod_{\{i,j\}\in E(H)} w_{\varphi(i), \varphi(j)}.$$ 
We can define $t(H, G_{\ww}) := \frac{\hom(H, G_{\ww})}{|V(G_{\ww})|^{|V(H)|}}$ analogously. 
If the edge weights $\ww$ of $G_\ww$ only take values in $\{0,1\}$, we recover the usual definitions of homomorphism numbers and densities. 

One of the central topics in extremal combinatorics is the study of
algebraic inequalities between homomorphism densities. 
For example, the famous Sidorenko conjecture \cite{Sidorenko} states that $t(H,G)-t(K_2,G)^{|E(H)|}\geq 0$ for any bipartite graph $H$ and any graph $G$. 
Since $t(H_i,G_{\ww})\cdot t(H_j,G_{\ww})=t(H_iH_j,G_{\ww})$ where $H_iH_j$ is the disjoint union of $H_i, H_j$, any polynomial inequality can be seen as a linear inequality.  Formally, define a {\it quantum graph} $f$  to be a finite formal linear combination of graphs $\sum_{1 \leq i \leq k} c_i H_i$ where $k \in \mathbb{N}_+$, $c_i \in \mathbb{R}$, and $H_i$'s are finite graphs \cite{FLS}. Many extremal combinatorics questions can be reformulated as asking whether \[t(f, G_{\ww}) := \sum c_i t(H_i, G_{\ww}) \geq 0\]  is valid for all graphs $G_{\ww}$ (for certain classes of edge weights $\ww$).

It is often useful to study the nonnegativity of quantum graphs using language and tools from {\it graph limits}, a deep theory recently developed by Lov\'asz et al. (see e.g., \cite{Lovaszbook, GL1, gl2}). Given a bounded measurable  function $W: [0,1]^2 \to \mathbb{R}$, one can define the homomorphism density of $H$ in $W$ as
$$t(H,W) := \int_{[0,1]^{|V(H)|}} \prod_{\{i,j\}\in E(H)} W(x_i, x_j)d x_1 \dots d x_{|V(H)|}.$$ 
There are three natural classes of graph limits \cite{Lovaszbook}. 
The first one is $\mathcal{W}$, the set of {\it kernels}, defined as the set of bounded symmetric measurable functions $W:[0,1]^2 \to \mathbb{R}$. The second  one is $\mathcal{W}_0$, the set of {\it graphons}, i.e., kernels where $W(x,y) \in [0,1]$ for all $x,y \in [0,1]^2$.
 The last one is $\mathcal{W}_1$, defined as the set of kernels $W$ with $\|W\|_\infty \leq 1$. More generally, we define $\mathcal{W}_a$ for some $a > 0$ as the set of kernels $W$ with $\|W\|_\infty \leq a$. Note that this notation is consistent with $\mathcal{W}_1$, but not with $\mathcal{W}_0$.
 
There are many  classical extremal combinatorics questions which deal not only with $\mathcal{W}_0$, but also with other kernels allowing negative values in the range. For example, it is known that $ t(\uKfour, W) + 3 t(\utriangleedge,W) + \frac{3}{4} t(\ucfour, W)+\frac{3}{4}t(\uHtwo,W) + \frac{3}{16} t(\uvedge\uvedge,W) \geq 0$ does not hold in general for kernels $W \in \mathcal{W}_{1/2}$ whereas $4 t(\upfour,W)+t(\uHtwo,W)+t(\uvedge\uvedge,W)\geq 0$ does \cite{K4,cycle}.  In general, the theory of Ramsey multiplicity or localness of graph inequalities can be rephrased as certifying the nonnegativity of polynomial inequalities for kernels $W$ in some class whose range includes negative values 
(e.g., \cite{openquestion, Lovaszlocal, hcommon, localSid, kcommon, Lovaszbook}).  
As a consequence of the results in \cite{GL1},  an answer to the following question raised by Lov\'asz would essentially give a complete solution to extremal combinatorics questions on graph homomorphism densities. 
\begin{problem}[Problem 20, Lov\'asz \cite{openquestion}]\label{p1}
Which  quantum graphs $f$ satisfy $t(f, W) \geq 0$ for every $W \in \mathcal{W}$ (or for every
$W \in \mathcal{W}_0$)?
\end{problem}

There have been many developments in solving some instances of Problem \ref{p1} (e.g., \cite{FLS,FA, LS}). In particular, many known results follow from writing the quantum graph $f$ as a sum of squares or, equivalently, through a sequence of cleverly applied  Cauchy-Schwarz inequalities \cite{FA, FLS, LS, Lovaszbook}.  An important question is the computational difficulty of deciding whether  a given polynomial inequality is valid.  
In  Problem 17 of  \cite{openquestion}, Lov\'asz asked, ``Does every algebraic inequality between the subgraph densities $t(F,W)$ that
holds for all $W \in \mathcal{W}_0$ follow from a finite number of semidefiniteness inequalities?" Hatami and Norin showed that the answer to this question is negative in their breakthrough paper \cite{HN11}. In fact, their main result in \cite{HN11} is that  the problem of determining the validity of a polynomial inequality between homomorphism densities for $W \in \mathcal{W}_0$ is undecidable. This shows that there is no unified algorithm or effective certificate to determine the positivity of $t(f, W)$ for any $W \in \mathcal{W}_0$, and negatively answers the following question raised by Lov\'{a}sz for the special case $W \in \mathcal{W}_0$.

\begin{problem}[Problem 21, Lov\'asz \cite{openquestion}]\label{p2}
Is it true that for every quantum graph
$f$ where $t(f, W) \geq 0$ for all $W \in \mathcal{W}$, there exist quantum graphs $g$ and $h$, each expressible as a sum of squares  of labeled quantum
graphs, so that $f + gf = h$?
\end{problem}

The undecidability result and the corresponding problem of effective certification of nonnegativity were left open for general weights. We show that the problem of determining whether a given graph density inequality is valid is undecidable for any set of kernels which contains all kernels with range $\{0,a\}$ for some non-zero $a\in \RR$. In particular, it is undecidable for each of the three natural classes of graph limits, $\mathcal{W}_0, \mathcal{W}$, and $\mathcal{W}_a$, thus generalizing the work of \cite{HN11}. We give a single unified proof for any kernel class $\mathcal{W}^*$ that contains all kernels with range $\{0,1\}$. 
The case of kernels containing all kernels with range $\{0,a\}$ follows quickly by rescaling the kernels by the factor of $\frac{1}{a}$, which also rescales the corresponding polynomial inequalities. We also note that by compactness \cite{GL1}, nonnegativity of a polynomial inequality on all kernels $W$ with range $\{0,1\}$ is equivalent to nonnegativity on all kernels $W\in \mathcal{W}_0$.

\begin{theorem}\label{thm:main}
For any set of kernels $\mathcal{W}^*$ which contains all kernels with range $\{0,1\}$ as a subset, the following problem is undecidable:
\begin{itemize}
    \item \textsc{Instance:} A positive integer $k$, finite graphs $H_1, \dots, H_k$, and real numbers $c_1, \dots, c_k$.
    \item \textsc{Question:} Does the inequality $\sum_{i=1}^k c_i t(H_i, W) \geq 0$ hold for all $W\in \mathcal{W}^*$?  
\end{itemize}
In particular, it holds when $\mathcal{W}^*$ is $\mathcal{W}$, $\mathcal{W}_0$ or $\mathcal{W}_a$, for $a>0$.
\end{theorem}

The above theorem quickly implies a negative answer to Problem \ref{p2} for any kernel class which contains all kernels with range $\{0,1\}$ as a subset. For nice choices of kernel classes, we have the following more general result on non-existence of certificates of nonnegativity.

\begin{corollary}\label{cor:rational}
There exists a quantum graph $f$ with $t(f,W)\geq 0$ for all $W \in \mathcal{W}$ (or $\mathcal{W}_1$) for which there are no quantum graphs $g,h$, each expressible as a sum of squares of labeled quantum graphs, such that $gf = h$. 
\end{corollary}
The solution of Problem \ref{p2} follows along the same lines as the proof of Corollary \ref{cor:rational}, which we will present in Subsection \ref{subsec:complete}. 

In addition, our result also implies the undecidability result for homomorphism numbers $\hom(H_i, G_{\ww})$ instead of densities $t(H_i, W)$ when the range of $\ww$ contains at least two values including $0$.  This follows from the fact that any polynomial inequality involving homomorphism densities $t(H_i, G_{\ww})$ is equivalent to a polynomial inequality involving homomorphism numbers $\hom(H_i, G_{\ww})$ by adding isolated vertices to $H_i$'s if necessary. Note that in the case when the edge weights $\ww$ has binary range $\{0,1\}$, an earlier result of Ioannidis and Ramakrishnan  \cite{IR95} showed undecidability of inequalities between homomorphism numbers $\hom(H_i, G)$. 

\subsection{Proof Idea and Challenges}

We first sketch the idea of Ioannidis and Ramakrishnan's short proof \cite{IR95} of the undecidability of inequalities between homomorphism numbers $\hom(H_i, G)$ as a motivation for the proof for homomorphism densities. As in \cite{HN11}, this proof is also deduced  from Matiyasevich's solution to Hilbert's tenth problem, stated below.
\begin{theorem*}[\cite{Mat70}]\label{thm:Mat}
Given a positive integer $k$ and a polynomial $p(x_1, \ldots, x_k)$ with integer coefficients, the problem of determining whether there exist $x_1, \ldots, x_k \in  \ZZ$ such that $p(x_1, \ldots, x_k) <0$ is undecidable.
\end{theorem*}
By changing $x_i$ to $-x_i$ when necessary, it is therefore also undecidable to determine whether a polynomial with integer coefficients is always nonnegative for $x_i$'s taking values in $\mathbb{N}$. Thus it suffices to show that for any polynomial with integer coefficients $p(x_1, \dots, x_k)$, there is a quantum graph $f$ such that $p(x_1, \dots, x_k) \geq 0$ for all $x_i \in \mathbb{N}$ if and only if $\hom(f, G) \geq 0$ for all $G$. Let $H_1, \dots,H_k$ be finite connected graphs with no homomorphisms from one to another and such that each $H_i$ has no non-trivial homomorphism to itself. It is not hard to show that such graphs exist. Since $\hom(H_i, G) \hom(H_j, G) = \hom(H_i H_j, G)$, there is a quantum graph $f$ such that for any graph $G$, 
\[p(\hom(H_1, G), \dots, \hom(H_k, G)) = \hom(f, G).\] Crucially, since $\hom(H_i, G) \in \mathbb{N}$, we have that $p\geq 0$ for any $x_1, \dots, x_k \in \mathbb{N}$ implies that $\hom(f, G) \geq 0$ for all $G$. On the other hand, for each $k$-tuple of values $a_1, \dots, a_k \in \mathbb{N}$,   there is a graph $G$ such that $t(H_i, G) = a_i$, for example by letting $G$ be the disjoint union of $a_i$ copies of $H_i$.

One challenge in generalizing this simple proof to show the undecidability of homomorphism density inequalities is that $t(H_i, G)$ is not necessarily an integer. 
In \cite{HN11}, Hatami and Norin 
used a result of Bollob\'{a}s \cite{Bol} that the convex hull of the set of all possible pairs of edge-triangle densities, i.e., pairs $(t(K_2, W), t(K_3, W))$ for $W \in \mathcal{W}_0$, is the convex hull of points $(1,1)$ and $\left(\frac{n-1}{n}, \frac{(n-2) (n-1)}{n^2}\right)$ for $n\in \mathbb{N}$. These extreme points thus provide the needed integer points. This \emph{integrality feature} alone does not lead to undecidability, since nonnegativity of univariate polynomials on integers is a decidable problem. 

Given a polynomial $p(x_1, \dots, x_k)$ in $k$ variables, starting from a particular base graph $F$, Hatami and Norin use a delicate and intricate construction of a quantum graph $f$ based on  different variations of blow-ups of $F$. As in the proof of the undecidability of homomorphism number inequalities, 
one needs $k$ ``independent" copies of the convex hull to ``plug in" $x_1, \dots, x_k$. Hatami and Norin achieve this by measuring some conditional graph densities, conditioned on the set of graph homomorphisms $\phi: V(F) \to V(G)$.  They construct the quantum graph $f$ so that for any graph $G$, \[t(f,G)=\sum_\phi c_\phi p^*(x_1(\phi), y_1(\phi),  \dots, x_k(\phi), y_k(\phi))\] where $c_\phi$'s are constants, $p^*$ is a polynomial whose nonnegativity is closely related to that of $p$, and $x_i(\phi)$ and $y_i(\phi)$'s are closely related to the density of $K_2$ and $K_3$ in some subgraphs $G_i(\phi)$ of $G$ depending on  $\phi$. When $p \geq 0$ for integer-valued variables, they show that $p^*\geq 0$ by the integrality feature of the aforementioned convex hull and the fact that each individual $\phi$ enables $k$ copies of this convex hull. A crucial fact is that since the weights $\ww$ of $G_{\ww}$ are nonnegative (in fact, they are in $\{0,1\}$), the constants $c_\phi$ are always nonnegative. These two facts  imply $t(f,G) \geq 0$ for any $G$.

There are several difficulties in extending this approach to more arbitrary weights $\ww$ in $G_{\ww}$, or equivalently to more general classes of kernels. First, much less is known about possible tuples of graph densities with $W \in \mathcal{W}$, and even less if we consider more general classes of kernels. There are certainly fewer valid inequalities when negative weights are allowed. For instance it is still an open question to characterize all graphs $G$ that such $t(G,W)\geq 0$ for all $W\in \mathcal{W}$ \cite{openquestion}. If we try to record again all pairs of edge-triangle densities $(t(K_2, W), t(K_3, W))$ but for $W \in \mathcal{W}$ instead of $\mathcal{W}_0$, then any point in $\mathbb{R}^2$ is achievable. Second, the construction that Hatami and Norin used heavily relies on $c_\phi$ being nonnegative, which is not the case in general for $W \in \mathcal{W}$. Lastly, the process to obtain multiple copies of the convex hull with the integrality feature relies on individual $\phi$'s, and cannot be used in our setting by the previous argument. 

Our proof strategy is to show that the convex hull of ratios of densities of some carefully chosen graphs has the integrality feature even when $W \in \mathcal{W}$. We then directly realize multiple copies of the convex hull by using an explicit family of graphs instead of going through a sum depending on $\phi$. Remarkably, the convex hull is the same regardless of the weights we use, either $\mathcal{W}_0$, $\mathcal{W}_1$ or $\mathcal{W}$. 
Another advantage of our proof technique is that we associate each polynomial to an explicit quantum graph.

\subsubsection{Main Results in Detail}
We use $\mathfrak{W}$ to denote the collection of sets of kernels that contain all kernels with range $\{0,1\}$ as a subset. 
We introduce the components of our quantum graphs. A {\it necklace} $N_{c,q}$ is a cycle of length $c$ where each edge is replaced by a clique of size $q$. (See Figure \ref{fig:N54P33} as an example, and Definitions \ref{def:q-ifi} and \ref{def:nc} for details). We study the convex hull of the profile recording the pairs of ratios of densities
$$ \left(\frac{t(N_{8,q},W)}{t(N_{4,q}, W)^2}, \frac{t(N_{12,q}, W)}{t(N_{4,q}, W)^3} \right)$$
as $W$ varies over $\mathcal{W}^*$ for some $\mathcal{W}^*\in\mathfrak{W}$, for which $t(N_{4,q},W)\neq 0$. We call the set of all possible pairs for all values of $q$ between $2$ and $\ell$, \emph{the $\ell$-necklace profile} and denote it as $\mathcal{D}_{\leq \ell}$ (see Definition \ref{def:D<l}).   Points in the $\ell$-necklace profile lie in $\mathbb{R}^{2l-2}$ since we are recording $l-1$ different pairs.  

Let $\mathcal{R} \subset \mathbb{R}^2$ be the convex hull of $(0,0)$ with the points $\left(\frac{1}{n},\frac{1}{n^2}\right)$ for $n \in \mathbb{N}$ (see Figure \ref{fig:regionR}).  We note that $\mathcal{R}$ is a compact convex set. Our first main ingredient is that the convex hull of the $\ell$-necklace profile $\mathcal{D}_{\leq \ell}$ is simply the direct product of $\ell-1$ copies of $\mathcal{R}$ inside $\mathbb{R}^{2\ell-2}$ for any $\ell \geq 2$ and any kernel class $\mathcal{W}^*\in \mathfrak{W}$.

\begin{theorem}\label{thm:profile}
Let $T\subset \mathbb{R}^{2\ell-2}$ be the convex hull of the $\ell$-necklace profile $\mathcal{D}_{\leq \ell}$ as $W$ varies over all kernels in $\mathcal{W}^*$, $\mathcal{W}^*\in \mathfrak{W}$. We have that
$$T= \mathcal{R}^{\ell-1}.$$
\end{theorem}
We will prove this result in Claim \ref{claim:convex} and Proposition \ref{prop:specialgraph} where the latter is the harder result. 
We observe that the set $\mathcal{R}$ is an affine-linear transformation of the convex hull of $(K_2,K_3)$ densities in $\mathcal{W}_0$ proved by Bollob\'{a}s \cite{Bol}, and therefore it has the same integrality feature. We can then follow the same reduction as Hatami and Norin to obtain our main theorem, Theorem \ref{thm:main}. In particular, we will prove that
deciding whether a polynomial expression in densities of necklaces $N_{c,q}$ is nonnegative for $W\in \mathcal{W}^*$ for any $\mathcal{W^*}\in \mathfrak{W}$ is undecidable. In particular, the result holds for $\mathcal{W}$, $\mathcal{W}_0$ and $\mathcal{W}_a$.


\section{Proof of Theorem \ref{thm:main}}

\subsection{Properties of graph profiles involving necklaces}\label{sec:profile}
The main result in this section is Theorem \ref{thm:profile}, which is the key to prove our main theorem (Theorem~\ref{thm:main}).
We first formalize the operation of replacing edges by cliques that was mentioned in the introduction in relation to necklaces. 

\begin{definition}\label{def:q-ifi}
Given some graph $G$ and some integer $q\geq 2$, the \emph{$q$-ification of $G$} is the graph obtained as follows: for every edge of $G$, add $q-2$ vertices that are all pairwise adjacent and that are all adjacent to the two vertices of the selected edge. In other words, each edge of $G$ gets replaced by a clique of size $q$. 
\end{definition}
Figure \ref{fig:N54P33} is a  $4$-ification of $G = C_5$. 
Necklaces can thus be defined through the $q$-ification of cycles. 

\begin{definition}\label{def:nc}
Let $c \geq 3, q \geq 2$ be positive integers. Let $N_{c,q}$ be the $q$-ification of the cycle of length $c$. We call $N_{c,q}$ the \emph{$q$-necklace} of length $c$.
\end{definition}

\begin{figure}[h!]
\centering
\includegraphics[scale=0.3]{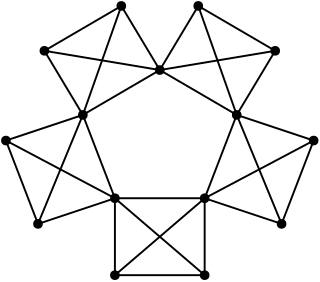} 
\caption{A $4$-necklace of length 5: $N_{5,4}$}\label{fig:N54P33}
\end{figure}

As it will soon become clear, to understand the density $t(N_{c,q}, W)$, we need to work with an auxiliary function corresponding to $W$. Before defining this auxiliary function, we first define conditional density. A \emph{rooted graph} is a graph with one or more vertices distinguished as roots. Suppose $H^{\bullet \dots \bullet}$ is a rooted graph with $k$ roots where $v_1, \dots, v_k$ are the roots and $v_{k+1}, \dots, v_{|V(H)|}$ are the remaining vertices of $H$. Fix $x_1, \dots, x_k \in [0,1]$, and define the conditional density
$$t_{x_1,\dots, x_k}(H^{\bullet \dots \bullet}, W) = \int_{[0,1]^{|V(H)|-k}} \prod_{\{i,j\}\in E(H)} W(x_i, x_j)d x_{k+1} \dots dx_{|V(H)|}.$$ 
In other words, this is the density of $H$ in $W$ where vertex $i$ is mapped to $x_i$ for each $1 \leq i \leq k$. 

\begin{definition}\label{def:MWq}
Given $W \in \mathcal{W}$, consider a symmetric measurable function $M_{W,q}: [0,1]^2 \to \RR$ such that $M_{W,q}(x,y) = t_{x,y}(K_q^{\bullet \bullet}, W)$ where $K_q^{\bullet \bullet}$ is the complete graph on $q$ vertices with two roots. 
\end{definition}

\begin{definition}\label{def:MWq2}
Given a graph $G$, we  define $M_{G, q} := M_{W_G, q}$ where $W_G$ is the graphon associated with $G$ defined as follows. Partition $[0,1]$ into $|V(G)|$ intervals $I_1, \dots, I_{|V(G)|}$ of equal length. Let $W_G(x,y)=1$ if and only if $x \in I_i$ and $y \in I_j$ and $(i,j) \in E(G)$; and $W_G(x,y)=0$ otherwise.
\end{definition} It can be easily seen that $t(H, G) = t(H, W_G)$. 
Trivially, by Definition \ref{def:MWq}, we have for any graph $G$ and any $W' \in \mathcal{W}$ that
\begin{equation}
    M_{G, 2} = W_G, \ \textup{and} \ M_{W',2} = W'. \label{eq:MG2}
\end{equation}

Since $M_{W, q}$ is bounded and symmetric, its spectrum is countable, with real eigenvalues $|\lambda_1| \geq |\lambda_2| \geq \dots$. In particular, $\sum_i |\lambda_i|^s<\infty$ for any positive integer $s$ (see e.g., \cite{Lovaszbook}).

\begin{lemma}\label{cor:relationhomeigenvalues}
Let $W \in \mathcal{W}$. Then $t(N_{c,q},W)=\sum_{i} \lambda_i^c$ where $\lambda_1, \lambda_2, \ldots$ are the eigenvalues of $M_{W,q}$. 
\end{lemma}

\begin{proof}
Since $M_{W,q}$ is again in $\mathcal{W}$, the result follows as  $t(N_{c,q}, W) = t(C_c, M_{W,q}) = \sum_i \lambda_i^n$. The last equality follows from a standard argument through the spectral decomposition of $M_{W,q}$ (e.g., see \cite{Lovaszbook}). 
\end{proof}

We now investigate the profiles of ratios of necklaces. 

\begin{definition}
For an integer $q \geq 2$, let $$\mathcal{D}_q:=\textup{cl}\left(\left \{ \left(\frac{t(N_{8,q},W)}{t(N_{4,q},W)^2}, \frac{t(N_{12,q},W)}{t(N_{4,q},W)^3}\right) \mid W \in \mathcal{W} \textup{ and }t(N_{4,q},W)\neq 0 \right\}\right).$$ 
For a particular $W$, we let $x_q(W):=\frac{t(N_{8,q},W)}{t(N_{4,q},W)^2}$ and $y_q(W):=\frac{t(N_{12,q},W)}{t(N_{4,q},W)^3}$. 
\end{definition}

We first note the following simple lemma where we obtain a bound on the boundary curve of $\mathcal{D}_q$.  This lower bound will serve the same purpose as the Goodman lower bound did in \cite{HN11}. 

\begin{lemma}\label{lem:Dq}
For any $W \in \mathcal{W}$, $0\leq x_q(W), y_q(W) \leq 1$. Equivalently, $\mathcal{D}_q\subseteq [0,1]^2$.  In addition, for any $(x,y)\in \mathcal{D}_q$, we have that $y\geq x^2$.
\end{lemma}
\begin{proof}
By Lemma~\ref{cor:relationhomeigenvalues}, we know that $x_q(W)=\frac{\sum_{i} \lambda_i^8}{\left(\sum_{i} \lambda_i^4\right)^2}$. Thus $x_q(W) \leq 1$  by the Cauchy-Schwarz inequality. 
The fact that $y_q(W) \in [0,1]$ follows from the same argument. 

We are left to show that $y_q(W) \geq x_q(W)^2$ holds for every $W \in \mathcal{W}$. By Lemma~\ref{cor:relationhomeigenvalues}, it suffices to show that $$\frac{\sum_{i} \lambda_i^{12}}{\left(\sum_{i} \lambda_i^4\right)^3}\geq \left(\frac{\sum_{i} \lambda_i^8}{\left(\sum_{i} \lambda_i^4\right)^2}\right)^2.$$ This is equivalent to showing that $\sum_{i} \lambda_i^4 \sum_{i} \lambda_i^{12} \geq  (\sum_{i} \lambda_i^8)^2$ which follows from H\"older's inequality. 
\end{proof}

We now want to show that a sparse but infinite set of points on $y=x^2$ is realizable for every $\mathcal{D}_q$. The following lemma and theorem build up to a result that serves the purpose of Bollob\'as' result used in \cite{HN11}.

\begin{lemma}\label{lem:firstrealization}
Let $q\geq 2$ be an integer. The points $(\frac{1}{r}, \frac{1}{r^2})$ are in $\mathcal{D}_q$ for every positive integer $r$.
\end{lemma}

\begin{proof}
Let $I_1, \dots, I_r$ be $r$ intervals of equal length that partition $[0,1]$. Let $W$ be a graphon where $W(x,y)=1$ if and only if both $x, y \in I_i$ for some $1 \leq i \leq r$. Equivalently, one can consider $W$ as the limit object of graphs consisting of the disjoint union of $r$ cliques of equal size. By Definition \ref{def:MWq}, $M_{W,q}(x,y)=0$ if $x,y$ are not in the same $I_i$ for any $1 \leq i \leq r$, and  $M_{W,q}(x,y) = \frac{1}{r^{q-2}}$  if and only if both $x,y\in I_i$ for some $1 \leq i \leq r$. Each block of $M_{W,q}$ (restricted to $I_i \times I_i$) has eigenvalues $\frac{1}{r^{q-2}}$ with multiplicity one and the remaining eigenvalues are zero. Thus the eigenvalues of $M_{W,q}$ are $\frac{1}{r^{q-1}}$  with multiplicity $r$ and zero for the remaining ones. Thus by Lemma~\ref{cor:relationhomeigenvalues}, $x_q(W) = \frac{1}{r}$ and $y_q(W) = \frac{1}{r^2}$.
\end{proof}

\begin{lemma}\label{lem:convex}
Let $q\geq 2$ be an integer. For every $W \in \mathcal{W}$ and every positive integer $r$, if $\frac{t(N_{8,q},W)}{t(N_{4,q}, W)^2} \in \left[\frac{1}{r+1}, \frac{1}{r}\right]$, then we have $$\frac{t(N_{12,q}, W)}{t(N_{4,q},W)^3}\geq L\left(\frac{t(N_{8,q},W)}{t(N_{4,q}, W)^2}\right)$$ where $$L(x):=\frac{2r+1}{r(r+1)}x-\frac{1}{r(r+1)}.$$ 
\end{lemma}

\begin{proof}
Let $\mathbf{X}$ be the space of $\mathbf{x} = (x_1, x_2, \dots)$ where $\mathbf{x} \geq \mathbf{0}$, $x_1 \geq x_2 \geq \ldots$, and $\mathbf{x} \in \ell^s$ for any positive integer $s$.  Let $e_j(\mathbf{x}):=\sum_{1 \leq i_1 < \ldots < i_j } x_{i_1}\cdots x_{i_j}$ be the $j$th elementary symmetric polynomial, and let $p_j(\mathbf{x}):=\sum_{i} x_i^j$ be the $j$th power sum. Note that when $\mathbf{x} \in \mathbf{X}$, $e_j(\mathbf{x}), p_j(\mathbf{x})$'s are well-defined. To prove the result, we first look at the convex hull of $e_2(\mathbf{x})$ and $e_3(\mathbf{x})$ when $e_1(\mathbf{x})=1$ before moving to the convex hull of the profile of $p_2(\mathbf{x})$ and $p_3(\mathbf{x})$ when $p_1(\mathbf{x})=1$ and then transforming this appropriately to get the desired result.

\begin{claim}For any $c_2, c_3\in \RR$, we have that $c_2e_2(\mathbf{x}) + c_3 e_3(\mathbf{x}) \geq 0$ for every $\mathbf{x}\in  \mathbf{X}$  and $e_1(\mathbf{x})=1$ if and only if $c_2e_2(\mathbf{x}) + c_3 e_3(\mathbf{x}) \geq 0$  for $x_1=\ldots = x_m = \frac{1}{m}$ and $x_{m+1}=x_{m+2}=\ldots=0$ for every integer $m\geq 1$.
\end{claim}
\begin{proof}
The forward direction is trivial since $x_1=\ldots = x_m = \frac{1}{m}$ and $x_{m+1}=x_{m+2}=\ldots=0$ satisfies $\mathbf{x} \geq \mathbf{0}$ and $e_1(\mathbf{x})=1$.

For the converse, suppose $c_2e_2(\mathbf{x}) + c_3 e_3(\mathbf{x}) \geq 0$  for $x_1=\ldots = x_m = \frac{1}{m}$ and $x_{m+1}=x_{m+2}=\ldots=0$ for every $m\in \NN$. Let $f(\mathbf{x}):=c_2e_2(\mathbf{x}) + c_3 e_3(\mathbf{x})$. It suffices to show that $$\min_{\substack{\mathbf{x} \in \mathbf{X},\\ e_1(\mathbf{x})=1}} f(\mathbf{x}) \geq 0.$$ 
Since $\{\mathbf{x}\in \mathbf{X} : e_1(\mathbf{x})=1\}$ is compact, we can assume that the minimum is attained on some $\mathbf{x} \in \mathbf{X}$, and among all optimal solutions, we pick one of largest $\ell^2$ norm. 

We prove the claim by local adjustment of this optimal solution $\mathbf{x}$ when fixing $x_i+x_j$ for some given $i < j$ where $x_i, x_j > 0$.
By the definition of $f$, there are functions $f_1, f_2, f_3$ that only depend on $\mathbf{x}^{i,j}$, i.e., $\mathbf{x}$ where we delete the entries $x_i, x_j$  so that
$$f(\mathbf{x})=x_ix_jf_1(\mathbf{x}^{i,j})+(x_i+x_j)f_2(\mathbf{x}^{i,j} )+f_3(\mathbf{x}^{i,j}).$$

If $f_1(\mathbf{x}^{i,j})\geq 0$, then by turning $x_i$ into $x_i + x_j$ and $x_j$ into zero, the value of $f$ does not decrease while the $\ell^2$ norm of the solution strictly increases, which is a contradiction after reordering $x_i$'s.
Thus $f_1(x_3, \ldots, x_n)<0$. Holding $x_i+x_j$ fixed, one can see that $f(\mathbf{x})$ is minimized when $x_i=x_j$. Similarly, we can show that all non-zero entries in $\mathbf{x}$ are equal. Thus there is a positive integer $m$ such that $x_1=\ldots =x_m = \frac{1}{m}$, as desired.
\end{proof}

Note that this implies that the convex hull of $\{(e_2(\mathbf{x}), e_3(\mathbf{x}))\mid \mathbf{x}\in \mathbf{X}, e_1(\mathbf{x}) = 1\}$ 
is the same as the convex hull of $\{(e_2(\mathbf{x}), e_3(\mathbf{x}))\mid x_1=\ldots = x_m = \frac{1}{m} \textup{ and } x_{m+1}= x_{m+2} = \ldots=0 \textup{ for some } m\in \NN\}$, or equivalently by evaluating, as the convex hull of $\{(\frac{m-1}{2m}, \frac{(m-1)(m-2)}{6m^2}) \mid m\in \NN\}$. Also note that this is equivalent to Bollob\'as' result \cite{Bol}, and that the exact profile can be fully understood as an easy corollary from Razborov's work \cite{RazTriangle}.

\begin{claim}\label{claim:2}The convex hull of $\{(p_2(\mathbf{x}), p_3(\mathbf{x}))\mid \mathbf{x}\in \mathbf{X}, p_1(\mathbf{x}) = 1\}$ is equal to the convex hull of $\{(\frac{1}{m}, \frac{1}{m^2})| m\in\NN\}$. 
\end{claim}
\begin{proof}
In general, through Newton's identities, we know that $p_2(\mathbf{x})=(e_1(\mathbf{x}))^2-2e_2(\mathbf{x})$, and $p_3(\mathbf{x})=(e_1(\mathbf{x}))^3-3e_1(\mathbf{x})e_2(\mathbf{x})+3e_3(\mathbf{x})$. Since we are setting $e_1(\mathbf{x})=p_1(\mathbf{x})=1$, we have $p_2(\mathbf{x})=1-2e_2(\mathbf{x})$ and $p_3(\mathbf{x})=1-3e_2(\mathbf{x})+3e_3(\mathbf{x})$, and so the convex hull of $\{(p_2(\mathbf{x}), p_3(\mathbf{x}))\mid \mathbf{x}\in \mathbf{X}, p_1(\mathbf{x}) = 1\}$ will be an affine transformation of the convex hull of $\{(e_2(\mathbf{x}), e_3(\mathbf{x}))\mid \mathbf{x}\in \mathbf{X}, e_1(\mathbf{x}) = 1\}$. Under this map, the points $(\frac{m-1}{2m}, \frac{(m-1)(m-2)}{6m^2})$ go to $(\frac{1}{m}, \frac{1}{m^2})$, and so the statement holds.
\end{proof}

\begin{claim}\label{claim:convex}
 The convex hull of $\{(x_q(W), y_q(W)) \mid W \in \mathcal{W}, t(N_{4,q}, W) \neq 0\}$ is contained in the convex hull of $\{(\frac{1}{m}, \frac{1}{m^2})| m\in\NN\}$. 
\end{claim}
\begin{proof}
Note for any real $\alpha \neq 0$ and any $W \in \mathcal{W}$, $x_q(W) =x_q(\alpha W)$, and $y_q(W) = y_q(\alpha W)$. Also notice that $t(N_{4,q}, \alpha W) =\alpha^{4\binom{q}{2}}t(N_{4,q},  W)$. Thus whenever $t(N_{4,q}, W)\neq 0$, we may assume $|t(N_{4,q}, W)| = 1$ by properly scaling $W$. Since $t(N_{4,q}, W) \geq 0$ by Lemma~\ref{cor:relationhomeigenvalues}, we assume $t(N_{4,q}, W) = 1$. 
Therefore the convex hull of $\{(x_q(W), y_q(W)) \mid W \in \mathcal{W}, t(N_{4,q}, W) \neq 0\}$ is the same as the convex hull of $\{(t(N_{8,q}, W), t(N_{12,q}, W)) \mid W \in \mathcal{W}, t(N_{4,q}, W) = 1\}$. 

Let $\bm{\lambda}=(\lambda_1, \lambda_2, \ldots)$ be the spectrum of $M_{W,q}$ where $|\lambda_1| \geq |\lambda_2| \geq \dots$. By letting $x_i=\lambda_i^4$ for every $i$, we have that $\mathbf{x} = (x_1, x_2, \dots) \in \mathbf{X}$. 
By Lemma~\ref{cor:relationhomeigenvalues},  $p_1(\mathbf{x})=p_4(\bm{\lambda})=t(N_{4,q},W)$, $p_2(\mathbf{x})=p_8(\bm{\lambda})=t(N_{8,q},W)$ and $p_3(\mathbf{x})=p_{12}(\bm{\lambda})=t(N_{12,q},W)$. So by Claim \ref{claim:2}, the convex hull of $\{(t(N_{8,q}, W), t(N_{12,q}, W)) \mid W \in \mathcal{W} \textup{ such that }t(N_{4,q}, W)=1\}$ is contained in the convex hull of $\{(\frac{1}{m}, \frac{1}{m^2})| m\in\NN\}$, as desired.
\end{proof}
This ends the proof of Lemma~\ref{lem:convex}.
\end{proof}

We give a name to the region in the statement of Lemma \ref{lem:convex}. Note that this region does not depend on the choice of $q$.

\begin{definition}\label{def:R}
Let $\mathcal{R}:=\{(x,y)\in [0,1]^2: y\leq x \textup{ and } y \geq \frac{2r+1}{r(r+1)}x-\frac{1}{r(r+1)} \textup{ for } \frac{1}{r+1}\leq x \leq \frac{1}{r} \textup{ for } r\in \NN\}$ (see Figure \ref{fig:regionR}). Observe that $\mathcal{D}_q\subseteq \mathcal{R}$ for every integer $q\geq 2$. 
Notice that $\mathcal{R}$ is the convex hull of the points $\{(1/n, 1/n^2), n \in \mathbb{N}\}$. 

\begin{center}
\begin{figure}[h]
\includegraphics[scale=0.5]{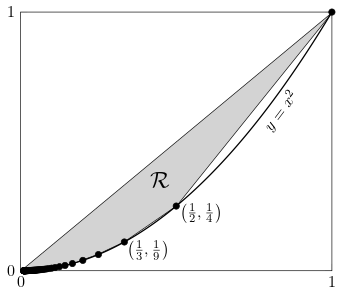}
\caption{Region $\mathcal{R}$ that contains $\mathcal{D}_q$ for every $q\geq 2$.}\label{fig:regionR}
\end{figure}
\end{center}
\end{definition}

\subsection{Obtaining $k$ independent copies of $\mathcal{D}_q$}\label{subsec:k}

In Lemma \ref{lem:firstrealization}, we realized an infinite sparse set of points on $y=x^2$ for every $\mathcal{D}_q$. 
In the next section, as explained in the proof outline, in order to simultaneously plug in independent values for each $x_i$ in $p(x_1, \dots, x_k) \in \mathbb{Z}[x]$, it will be useful to simultaneously realize  $(\frac{1}{r_q},\frac{1}{r_q^2})\in \mathcal{D}_q$ with a single construction for different positive integer $r_q$'s for every $2 \leq q \leq \ell$ (in the sense of Proposition \ref{prop:specialgraph}). To do so, we now use $q$-ification again on a construction due to Alon \cite{Alon}. 

Recall that $x_q(W):=\frac{t(N_{8,q},W)}{t(N_{4,q},W)^2}$ and $y_q(W):=\frac{t(N_{12,q},W)}{t(N_{4,q},W)^3}$

\begin{definition}\label{def:D<l}
For $\ell\in \NN$ such that $\ell \geq 2$, let \begin{align*}\mathcal{D}_{\leq \ell}:=\textup{cl}\left(\left\{ \left(
x_2(W), y_2(W), x_3(W), y_3(W), \dots, x_\ell(W), y_\ell(W)
\right) :  
 W \in \mathcal{W},   t(N_{4,q},W)\neq 0 \ \forall 2 \leq q \leq \ell \right\}\right).\end{align*} 
\end{definition}

Our goal in this subsection is to prove the following proposition. 
\begin{proposition}\label{prop:specialgraph}
For any set of positive integers $r_2, \ldots, r_\ell$, $(\frac{1}{r_2}, \frac{1}{r_2^2},\frac{1}{r_3}, \frac{1}{r_3^2}, \ldots, \frac{1}{r_\ell}, \frac{1}{r_\ell^2})\in \mathcal{D}_{\leq \ell}$. 
\end{proposition}

Note that Theorem~\ref{thm:profile} follows from Proposition~\ref{prop:specialgraph}.

To prove Proposition \ref{prop:specialgraph}, we need a construction of $W$ for any given $r_2, \dots, r_\ell$. The starting point of our construction $W$ is a disjoint union of blocks, each of the style described in the next definition.

\begin{definition}
Let $A(k,2)$ be an (unweighted) triangle-free graph with $n$ vertices that is $d$-regular and where the second largest eigenvalue (in absolute value) of the adjacency matrix of $A(k,2)$ is $\lambda$, where $n=\Theta(2^{3k})$, $d=\Theta(2^{2k})$, and $\lambda=\Theta(2^k)$. Here the constants in $\Theta$ are absolute constants. For an integer $q\geq 3$, let $A(k,q)$ be the  $q$-ification of $A(k,2)$.

The graphs $A(k,2)$ are known as $(n,d,\lambda)$-graphs in the literature, and were first constructed by Alon in \cite{Alon}. 
 Since $A(k,2)$ has $\frac{nd}{2}=\Theta(2^{5k-1})$ edges and since the $q$-ification introduces $q-2$ vertices for every edge of $A(k,2)$, $A(k,q)$ has $n+(q-2)\frac{dn}{2}=\Theta(2^{3k}+(q-2)2^{5k-1})$ vertices. The original vertices coming from $A(k,2)$ have degree $d(q-1)$, and the new vertices have degree $q-1$. Further, the original vertices are each contained in $d$ $q$-cliques, whereas the new vertices are contained in exactly one $q$-clique. 
\end{definition}

In the proof of Proposition \ref{prop:specialgraph}, each block of our construction $W$ will look like $A(k_s, s)$ for some appropriately chosen $k_s$ for each $2 \leq s \leq \ell$. Thus to understand $t(N_{4,q}, W), t(N_{8, q}, W)$ and $t(N_{12, q}, W)$, we need to understand $M_{A(k_s, s), q}$ (as defined in Definition~\ref{def:MWq2}), and bound its eigenvalues to apply Lemma~\ref{cor:relationhomeigenvalues}. Recall that $W_G$ is the graphon corresponding to graph $G$. 

\begin{lemma}\label{lem:qadj}
For  $s, q \geq 2$, suppose $A(k,s)$ has $n_s$ vertices. Then 
$M_{A(k,s),q}=\binom{s-2}{q-2}\frac{W_{A(k,s)}}{n_s^{q-2}}$. In particular, this means that $M_{A(k,q),q}=W_{A(k,q)}$.
\end{lemma}

\begin{proof}
Since $A(k,2)$ contains no clique of size 3, the only $s$-cliques in $A(k,s)$ are those that correspond to edges in $A(k,2)$. In fact, these $s$-cliques are the only maximal cliques in $A(k,s)$ Therefore, every edge of $A(k,s)$ is in exactly $\binom{s-2}{q-2}$ $q$-clique. The proof is complete by (\ref{eq:MG2}) that $M_{A(k,s),2} = W_{A(k,s)}$.
\end{proof}

\begin{lemma}\label{lem:secondeig}
The top eigenvalue of the adjacency matrix of graph $A(k,s)$ is $\Theta(2^{2k})$, while all the other eigenvalues are at most $O(2^k)$. Here the asymptotic is as $k \to \infty$ and $s$ is a fixed constant. 
\end{lemma}
\begin{proof}
 Since $A(k,2)$ is an induced subgraph of $A(k,s)$, by the eigenvalue interlacing theorem, the largest eigenvalue of $A(k,s)$ is at least that of $A(k,2)$, which is $\Theta(2^{2k})$. On the other hand, the largest eigenvalue is also at most the maximum degree of $A(k,s)$, which again is $\Theta(2^{2k})$. 
 
 We now prove the second claim. For any edge $e = (u,v)$ in $A(k,2)$, let $U_e$ be the $s-2$ vertices created for $e$ in $A(k, s)$. Let $\lambda$ be an eigenvalue of $A(k,s)$ and $\mathbf{v}$ its eigenvector and assume $|\lambda| \gg \Theta(1)$. For any $w \in U_e$, we have \begin{equation}
     \lambda \mathbf{v}(w) = \sum_{w' \in U_e}\mathbf{v}(w')  - \mathbf{v}(w)+ \mathbf{v}(v)+\mathbf{v}(u). \label{eq:eigen1}
 \end{equation} 
 By subtracting two such equations if $|U_e| \geq 2$, we obtain $\mathbf{v}(w) = \mathbf{v}(w')$ for any $w, w' \in U_e$ since $\lambda \neq -1$. Plugging this back into (\ref{eq:eigen1}) and using $|U_e| = s-2$, we have $(\lambda-s+3)\mathbf{v}(w) = \mathbf{v}(v)+ \mathbf{v}(u)$ for any $w \in U_e$, implying that the values $\mathbf{v}(w')$ are the same for all $w' \in U_e$. Furthermore, we have that
 \begin{equation}
     \sum_{w \in U_{(u,v)}}\mathbf{v}(w) = \frac{s-2}{\lambda-s+3}(\mathbf{v}(v)+ \mathbf{v}(u)).  \label{eq:eigen2}
 \end{equation}

Consider any original vertex $v$ in $A(k,2)$. We then have
\[
\lambda \mathbf{v}(v) = \sum_{u: (u,v)\in E(A(k,2)) } \left(\mathbf{v}(u)  + \sum_{w \in U_{(u,v)}} \mathbf{v}(w)\right) =\sum_{u: (u,v)\in E(A(k,2)) } \left(\mathbf{v}(u) + \frac{s-2}{\lambda-s+3}(\mathbf{v}(v)+ \mathbf{v}(u))\right)
\]
where the second equality follows from (\ref{eq:eigen2}). By rearranging and recalling that each vertex in $A(k,2)$ has degree $d = \Theta(2^{2k})$, we have that
\begin{equation}
    \left(\lambda - \frac{d(s-2)}{\lambda-s+3}\right) \mathbf{v}(v) =\left( 1+ \frac{s-2}{\lambda-s+3}\right) \sum_{u: (u,v)\in E(A(k,2)) } \mathbf{v}(u).
\end{equation}
Moving $1+\frac{s-2}{\lambda-s+3}$ to the left, this quantity is non-zero since $|\lambda|$ is large. Thus \[ f(\lambda) := \frac{\lambda - \frac{d(s-2)}{\lambda-s+3}}{ 1+ \frac{s-2}{\lambda-s+3}} \]
has to be an eigenvalue of $A(k,2)$,  with the same eigenvector as $\mathbf{v}$ when restricted to vertices in $A(k,2)$. 

We can now prove the second part of the lemma by contradiction. 
If $\lambda$ is not the largest eigenvalue of $A(k,s)$, but $\frac{2^k}{\lambda} = o(1)$ as $k \to \infty$, then $f(\lambda) = \Theta(\lambda) \gg 2^k$ since $p$ is a fixed constant and $d = \Theta(2^{2k})$. However, the only eigenvalue of $A(k,2)$ much larger than $\Theta(2^k)$ is the top eigenvalue, which is $d$. Thus $f(\lambda)=d$. Rearranging the equation $f(\lambda)=d$, any eigenvalue $\lambda$ satisfying $\frac{2^k}{\lambda} = o(1)$ is a root of $\lambda(\lambda - s +3) - d(s-2) = d(\lambda - s + 3 + s-2)$. The product of the roots of this quadratic equation is $-d(s-2)-d$. Since we have shown the largest eigenvalue of $A(k,s)$ is at least $d$, the magnitude for the other root is at most $\frac{d(s-2)+d}{d} = O(1)$.
Since $A(k, s)$ is connected, the largest eigenvalue has multiplicity one. This leads to a contradiction to our assumption on $\lambda$. 
\end{proof}

\begin{corollary}\label{cor:eigs}
Fix $s \geq q \geq 2$. The top eigenvalue of $M_{A(k,s),q}$ is $\Theta(2^{2k-5k(q-1)})$, and all the other eigenvalues are at most $O(2^{k-5k(q-1)})$.
\end{corollary}
\begin{proof}
This is a direct consequence of Lemmas \ref{lem:qadj} and \ref{lem:secondeig} and the fact that $A(k,s)$ has $n_s = \Theta(2^{5k})$ vertices. Here notice that the eigenvalues of $M_{A(k,s),2}= W_{A(k,s)}$ are the ones of the adjacency matrix of $A(k,s)$ scaled by $\frac{1}{n_s}$. 
\end{proof}

\begin{lemma}\label{lem:eigenvaluesremaining}
All but at most $O(2^{3k})$ eigenvalues of $A(k,s)$ are $O(1)$. Thus all but at most $O(2^{3k})$ eigenvalues of $M_{A(k,s),q}$ are $O(2^{-5k(q-1)})$.
\end{lemma}
\begin{proof}
Let the eigenvalues of $A_{k,s}$ 
be $\lambda_1\geq \lambda_2 \geq \ldots $. Let $n$ be the number of vertices in $A(k,2)$, which is $\Theta(2^{3k})$. By the interlacing theorem applied to adjacency matrices, we know that $\lambda_{1+n}$ is at most the largest eigenvalue of the adjacency matrix of the graph obtained by removing the $n$ original vertices in $A(k,2)$ vertices from $A(k,s)$. This induced subgraph  consists of the disjoint union of cliques of size $q-2$, and thus has eigenvalues that are at most $q-3$.  So $\lambda_{1+n}\leq q-3$. Similarly, each eigenvalue in $K_{q-3}$ is at least $-1$, thus all but at most $n$ eigenvalues are at least  $-1$. 
\end{proof}

We are now ready to prove Proposition \ref{prop:specialgraph}.
\begin{proof}[Proof of Proposition \ref{prop:specialgraph}]
For each $2 \leq s \leq \ell$, 
let $k_s = (\ell-s + 1)K$ for some large positive integer $K$. All the asymptotics in this proof are as $K \to \infty$, and the constants are in terms of $\ell$ and $r_s$'s. Let graphon $W$ correspond to the disjoint union of $G_2$, $G_3$, $\ldots$, $G_\ell$ where $G_s$ consists of $r_s$ disjoint copies of $A(k_s,s)$. Each $A(k_s,s)$ has $n_s  = \Theta(2^{5k_s})$ vertices, and let $N = \sum_s r_s n_s$ be the total number of vertices in the graph.
Thus each block in $M_{W,q}$ corresponding to $A(k_s, s)$ is $M_{A(k_s,s),q}$ scaled by $\frac{n_s^{q-2}}{N^{q-2}}$. 
When fixing $s, q$, let $\mu_{s,q,i}$ be the $i$-th largest eigenvalue of $M_{A(k_s,s),q}$ multiplied by $n_s^{q-2} \cdot n_s$.

Thus by Lemma~\ref{cor:relationhomeigenvalues}, 
\begin{align*}
x_q(W)&=\frac{\sum_{s=2}^\ell r_s\left(\sum_{i} \mu_{s,q,i} ^8\right)}{\left(\sum_{s=2}^\ell r_s\left(\sum_i \mu_{s,q,i}^4\right)\right)^2}=\frac{\sum_{s=q}^\ell r_s\left(\sum_{i} \mu_{s,q,i}^8\right)}{\left(\sum_{s=q}^\ell r_s\left(\sum_i \mu_{s,q,i}^4\right)\right)^2}
\end{align*} 
The last equality holds because  when $s < q$, $M_{A(k,s),q}$ is zero by Lemma \ref{lem:qadj}. 

When $\ell \geq s \geq q \geq 2$, by  Corollary \ref{cor:eigs}, $\mu_{s,q,1}= \Theta(2^{k_s(2-5(q-1))}) n_s^{q-1} = \Theta(2^{2{k_s}})$.  Furthermore, by Corollary \ref{cor:eigs} and Lemma \ref{lem:eigenvaluesremaining}, $\sum_i \mu_{s,q,i}^4 = \mu_{s,q,1}^4 +  \sum_{i \geq 2}\mu_{s,q,i}^4 = \mu_{s,q,1}^4 + O(2^{3k_s})\cdot  O(2^{4k_s(1-5(q-1))}) n_s^{q-1}  + n_s O(2^{-20k_s(q-1)}) = (1+o(1)) \mu_{s,q,1}^4$. In addition, when $s > q$, $\mu_{s,q,1}^4= \Theta(2^{8k_s}) = o(2^{8k_q}) = o(\mu_{q,q,1}^4)$. Thus the denominator of $x_q(W)$ is $(\sum_{p=q}^\ell (1+o(1))r_s\mu_{s,q,1}^4)^2 = (1+o(1))r_s^2\mu_{q,q,1}^4$. Similarly, by the same computation, the numerator of $x_q(W)$ is $\sum_{s=q}^\ell (1+o(1))r_s\mu_{s,q,1}^8 = (1+o(1))r_s\mu_{q,q,1}^8$. Thus $x_q(W) = \frac{1+o(1)}{r_q}$. Similarly, $y_q(W) = \frac{1+o(1)}{r_q^2}$. Again here $o(1)$ is when $K \to \infty$.  Since this result holds for any $2 \leq q \leq \ell$, the claim follows. 
\end{proof}

We end this section by highlighting the fact that even though our profiles are over $W \in \mathcal{W}$,  all our constructions use $W \in \mathcal{W}_0$, and thus the same results hold had we started with profiles for $W \in \mathcal{W}^*$ for any $\mathcal{W}^* \in \mathfrak{W}$.

\subsection{Completion of the proof}\label{subsec:complete}
By Matiyasevich's solution to Hilbert's tenth problem \cite{Mat70}, determining the nonnegativity of a polynomial with integer-valued variables is undecidable. We follow the same strategy as Hatami and Norin in \cite{HN11} to reduce our problem to the decidability of polynomials with integer variables. We use the integrality feature of the graph profiles proved in Lemma~\ref{lem:firstrealization}  and Proposition~\ref{prop:specialgraph} of Subsection \ref{sec:profile}. 

\begin{lemma}\label{lem:undecidableintegrality}
Given a positive integer $\ell\geq 7$ and a polynomial $p(x_2, \ldots, x_\ell)$ with integer coefficients, the problem of determining whether there exist $x_2, \ldots, x_\ell \in \{\frac{1}{n} | n\in \NN\}$ such that $p(x_2, \ldots, x_\ell) <0$ is undecidable.
\end{lemma}

\begin{proof}
Consider a polynomial $\bar{p}(y_2, \ldots, y_\ell)$ with integer coefficients. Note that $\bar{p}(y_2, \ldots, y_\ell) \geq 0$ for all $y_2, \ldots, y_\ell\in \NN$ if and only if the polynomial with integer coefficients $$p(x_2, \ldots, x_\ell):=\left(\prod_{q=2}^{\ell}x_q^{\textup{deg}(\bar{p})}\right)\bar{p}\left(\frac{1}{x_2}, \ldots, \frac{1}{x_\ell} \right)$$ is nonnegative for all $x_2, \ldots, x_\ell\in \{\frac{1}{n}: n\in \NN\}$. Hence, the problem is undecidable.
\end{proof}

Just as in \cite{HN11}, we need to relate the nonnegativity of a polynomial with variables in $\{\frac{1}{n} | n\in \NN\}$ to the nonnegativity of polynomials involving homomorphism densities. By Lemma \ref{lem:convex}, the convex hull $\mathcal{D}_q$ has the integrality feature for every $2\leq q \leq \ell$. 
To use this fact, we first need to create an auxiliary polynomial as in Lemma 5.4 in \cite{HN11} so that the original polynomial is nonnegative on $\{1/n: n\in \mathbb{N}\}$ if and only if the auxiliary polynomial is nonnegative on $\mathcal{R}^{\ell-1}$.
We modify slightly the construction of the auxiliary polynomial presented in \cite{HN11}, and reproduce the proof in the Appendix for completion.

\begin{lemma}\label{lem:pbar}
Let $p$ be a polynomial in variables $x_2, \ldots, x_\ell$. Let $M$ be the sum of the absolute values of the coefficients of $p$ multiplied by $100 \textup{deg}(p)$. Define $\bar{p}\in \RR[x_2, \ldots, x_\ell, y_2, \ldots, y_\ell]$ as $$\bar{p}:=p\prod_{q=2}^\ell x_q^6 + M \left(\sum_{q=2}^\ell y_q - x_q^2\right).$$ Then the following are equivalent:
\begin{enumerate}
\item $\bar{p}(x_2, \ldots, x_\ell, y_2, \ldots, y_\ell)<0$ for some $x_2, \ldots, x_\ell, y_2, \ldots, y_\ell$ with $(x_q, y_q)\in \mathcal{R}$ for every $2\leq q \leq \ell$;
\item $p(x_2, \ldots, x_\ell)<0$ for some $x_2, \ldots, x_\ell \in \{\frac{1}{n}:n\in \NN\}$. 
\end{enumerate}
\end{lemma}

\begin{lemma}
Given a polynomial $p$ in $\ell-1$ variables, there is a quantum graph $f(p)$ such that for any $W \in \mathcal{W}$, 
\begin{equation}
    t(f(p), W):= \bar{p}\left(\frac{t(N_{8,2},W)}{t(N_{4,2},W)^2}, \ldots, \frac{t(N_{8,\ell},W)}{t(N_{4,\ell},W)^2}, \frac{t(N_{12,2},W)}{t(N_{4,2},W)^3}, \ldots, \frac{t(N_{12,\ell},W)}{t(N_{4,\ell},W)^3}\right)\prod_{q=2}^\ell t(N_{4,q},W)^{3\textup{deg}(p)}. \label{eq:quantum}
\end{equation}
\end{lemma}
\begin{proof}
The right hand side of (\ref{eq:quantum}) is a polynomial in $t(N_{4,q}, W), t(N_{8,q},W), t(N_{12,q},W)$ for $2\leq q \leq \ell$. Thus there is a quantum graph $f(p)$ where (\ref{eq:quantum}) holds.
\end{proof}

\begin{lemma}\label{lem:equiv}
The following two statements are equivalent. 
\begin{itemize}
    \item $p(x_2, \ldots, x_\ell)<0$ for some $x_2, \ldots, x_\ell \in \{\frac{1}{n}:n\in \NN\}$. 
    \item $t(f(p),W) < 0$ for some $W \in \mathcal{W}$. 
\end{itemize}
\end{lemma}
\begin{proof}
If $p(x_2, \ldots, x_\ell)\geq 0$ for all $x_2, \dots, x_\ell \in \{\frac{1}{n}: n\in \mathbb{N}\}^{\ell-1}$, then 
$\bar{p}(x_2, \ldots, x_\ell, y_2, \ldots, y_\ell)\geq 0$ for every $x_2, \ldots, x_l, y_2, \ldots, y_\ell$ such that $(x_q,y_q)\in \mathcal{R}$ for every $2\leq q \leq \ell$ by Lemma \ref{lem:pbar}. Therefore $t(f(p),W)\geq 0$ for every $W \in \mathcal{W}$ by (\ref{eq:quantum}) and Lemma \ref{lem:firstrealization}. 

On the other hand, if $p(x_2', \ldots, x_\ell') < 0$ where for each $2 \leq q \leq \ell$, we have that $x_q' =  \frac{1}{n_q}$ for some $n_q \in \mathbb{N}$, then by the construction of $\bar p$, $\bar{p}(\frac{1}{n_2}, \ldots, \frac{1}{n_\ell}, \frac{1}{n_2^2}, \ldots, \frac{1}{n_\ell^2})< 0$. By Proposition \ref{prop:specialgraph}, there is a sequence of $W_n$'s such that $\frac{t(N_{8,q},W_n)}{t(N_{4,q},W_n)^2} \rightarrow \frac{1}{n_q}$ and $\frac{t(N_{12,q},W_n)}{t(N_{4,q},W_n)^3}\rightarrow \frac{1}{n_q^2}$ for each $2 \leq q \leq \ell$ as $n\rightarrow \infty$. Thus again by (\ref{eq:quantum}) and the continuity of polynomial $\bar p$, there is a $W$ such that $t(f(p),W) < 0$.
\end{proof}

We can now finish the proof of Theorem \ref{thm:main}.
\begin{proof}[Proof of Theorem \ref{thm:main}]
This is a direct consequence of Lemma \ref{lem:equiv} and Theorem \ref{thm:Mat}. 
\end{proof}
Again, note that the same holds if we look at the version of this theorem where we replace graph homomorphisms by homomorphism numbers.

\begin{proof}[Proof of Corollary \ref{cor:rational}]
This follows from the same argument that Lov\'asz presented in \cite{Lovaszbook} for the positivstellensatz of graphs. Consider two Turing machines where the input is a quantum graph $f$ with rational coefficients. One machine is searching for sums of squares  $g, h$ where $gf =h$; the other machine is searching for graphs $G_{\ww}$ where the range of $\ww$ is $\mathbb{Z}$. Both are well-defined algorithms. If either machine halts, we know $f \geq 0$ or not. However, by our undecidability result Theorem \ref{thm:main}, there must be an $f$ where both machines never halt. 
\end{proof}

 \bibliographystyle{alpha}
\bibliography{undecidabilityreferences}

\appendix
\section{Proof of Lemma \ref{lem:pbar}}
\begin{proof}
If $(2)$ holds, then for each $x_q$ we have $(x_q, x_q^2)\in \mathcal{R}$ by the definition of $\mathcal{R}$, and setting $y_q:=x_q^2$ gives $\bar{p}(x_2, \ldots, x_\ell, y_2, \ldots, y_\ell)=p(x_2, \ldots, x_\ell)\prod_{q=2}^\ell x_q^6<0$. Therefore, $(2)$ implies $(1)$. 

Suppose now that $(1)$ holds. Observe that decreasing $y_q$ decreases the value of $\bar{p}$, and we can thus assume without loss of generality that $y_q$ is as small as possible while still having $(x_q, y_q)\in \mathcal{R}$, i.e., $y_q=L(x_q)$. So we have $$\tilde{p}(x_2, \ldots, x_\ell):=\bar{p}(x_2, \ldots, x_\ell, y_2, \ldots, y_\ell)=p(x_2, \ldots, x_\ell)\prod_{q=2}^\ell x_q^6+ M\left(\sum_{q=2}^\ell L(x_q) - x_q^2\right)<0.$$ For every $2 \leq q \leq \ell$, let $r_q$ be a positive integer such that $x_q \in \left[\frac{1}{r_q+1}, \frac{1}{r_q}\right]$. Fixing $r_2, \ldots, r_\ell$, we may assume that $x_2, \ldots, x_\ell$ are chosen in the corresponding intervals to minimize $\tilde{p}$. We claim that we then have that $x_q \in \{\frac{1}{n}:n\in \NN\}$ for every $2 \leq q \leq \ell$. Suppose not: suppose that $x_{q^*} \in (\frac{1}{r_{q^*}+1}, \frac{1}{r_{q^*}})$ for some $q^*$. By the choice of $x_{q^*}$, we have that $\frac{\partial \tilde{p}}{\partial x_{q^*}}=0$. We will derive a contradiction by a sequence of observations.

First note that since each $x_q\in (0,1]$ and $\textup{deg}(p)\geq 1$, we have that
\begin{align*}
 &\bigg|\frac{\partial}{\partial x_{q^*}}\left( p(x_2, \ldots, x_\ell) \prod_{q=2}^\ell x_q^6\right)\bigg| = p \cdot  6x_{q^*}^5\cdot \prod_{q\in [\ell]\backslash\{1,q^*\}} x_q^6 + \frac{\partial p}{\partial x_{q^*}} \cdot \prod_{q=2}^\ell x_q^6 
 \leq \sum_i |c_i| \cdot 6x_{q^*}^5 \cdot 1 + \textup{deg}(p)\sum_{i} |c_i|\cdot x_{q^*}^6\\
& \leq  7 \textup{deg}(p) \sum_{i} |c_i| x_{q^*}^5
=7 \frac{M}{100} x_{q^*}^5
\leq \frac{M}{12 r_{q^*}^5}
\end{align*}
where $\sum_i |c_i|$ is the sum of the absolute values of the coefficients of $p$.

Therefore, since $\frac{\partial \tilde{p}}{\partial x_{q^*}}=0$, we have
$\frac{1}{12r_{q^*}^5} \geq |L'(x_{q^*})-2x_{q^*}|=\bigg| 2z - 2x_{q^*}\bigg|$  
where $z=\frac{r_{q^*}+\frac{1}{2}}{r_{q^*}(r_{q^*}+1)}$. Then the previous inequality can be rewritten as $$|z-x_{q^*}| \leq \frac{1}{24r_{q^*}^5}.$$ Now note that $L'(x)-2x=\frac{2r+1}{r(r+1)}-2x$ is monotone on the interval $\left[\frac{1}{r_{q^*}+1},{\frac{1}{r_{q^*}}}\right]$ which contains both $x_{q^*}$ and $z$, and that this expression yields $0$ when $x=z$. 
It follows that 
\begin{align*}
L(x_{q^*})-x^2_{q^*} &\geq (L(z)-z^2)-|(L(x_{q^*})-x_{q^*}^2)'||z-x_{q^*}|\\
& \geq \frac{1}{4r_{q^*}^2(r_{q^*}+1)^2}-2\left(\frac{1}{24r_{q^*}^5}\right)^2
\geq \frac{1}{16r_{q^*}^4}-2\left(\frac{1}{24r_{q^*}^5}\right)^2=\frac{1}{18r_{q^*}^4}.
\end{align*}

Finally, note that $\sum_{q=2}^{\ell} L(x_q)-x_q^2 \geq L(x_{q^*})-x_{q^*}^2$ since $L(x_q)-x_q^2 \geq 0$ for every $q$. So, putting everything together, we thus have that
$$\tilde{p}(x_2, \ldots, x_\ell) \geq -\frac{M}{100} x_{q^*}^6+M(L(x_{q^*})-x_{q^*}^2) \geq M \left( \frac{1}{18r_{q^*}^4} - \frac{1}{100 r_{q^*}^6}\right) > 0,$$ which is a contradiction. Therefore, the claim that $x_q\in \{\frac{1}{n} : n\in \NN\}$ for every $2\leq q \leq \ell$ holds. We therefore have that 
$0>\tilde{p}(x_2, \ldots, x_{\ell}) = p(x_2, \ldots, x_\ell) \prod_{q=2}^\ell x_q^6,$ which shows that $(2)$ holds since $\prod_{q=2}^\ell x_q^6\geq 0$ . 
\end{proof}

\section{Alternative proof}
We provide an alternative proof to Theorem \ref{thm:main}. The advantage of this proof is that it is easier to obtain $k$ independent copies of the convex hull, although the description of the quantum graph is slightly more complicated. 

In some steps of this proof, it is easier to work with {\it induced} copies of $H$ in $G$. Given graphs $H$ and $G$, an {\it induced homomorphism} from $H$ to $G$ is a mapping $V(H) \to V(G)$ that preserves both adjacency and non-adjacency. Given a graph $H$ and $W \in \mathcal{W}$, we can define the {\it induced density} as
\[
\tind(H, W) := \int_{[0,1]^{|V(H)|}} \prod_{\{i,j\}\in E(H)} W(x_i, x_j)\prod_{\{i,j\}\notin E(H)} (1-W(x_i, x_j))d x_{1} \dots dx_{|V(H)|}.
\]
Thus if $W_G$ is the graphon associated with a simple graph $G$, $\tind(H, W_G)$ is essentially the density of induced copies of $H$ in $G$. We can similarly define conditional induced density. Suppose $H^{\bullet \dots \bullet}$ has $k$ roots, then 
$$t_{\textup{ind}, x_1,\dots, x_k}(H^{\bullet \dots \bullet}, W) := \int_{[0,1]^{|V(H)|-k}} \prod_{\{i,j\}\in E(H)} W(x_i, x_j) \prod_{\{i,j\}\notin E(H)} (1-W(x_i, x_j))  d x_{k+1} \dots dx_{|V(H)|}.$$

Let $\mathcal{H}$ be the set of connected rooted graphs $H^{\bullet \bullet}$ where (1) the two roots are adjacent, and (2) the graph is symmetric  with respect to  the two roots, i.e., there is an induced homomorphism $V(H) \to V(H)$ which maps the two roots to one another. Let $N_{c, H^{\bullet \bullet}}$ be the  graph obtained by gluing each edge $e$ of the cycle $C_c$  to an induced copy of $H$ by identifying the two roots of $H^{\bullet \bullet}$ with the end vertices of $e$. 
Let $\bN_{c, H^{\bullet \bullet}}$ be the set of all possible graphs that contain $N_{c, H^{\bullet \bullet}}$ as a spanning subgraph, and where there can be additional edges formed by any pair of vertices belonging to copies of $H$ that were glued to different edges of $C_c$.  To lighten the notation, we let  
\[
\tind(\bN_{c,H^{\bullet \bullet}}, W) := \sum_{F \in \bN_{c,H^{\bullet \bullet}}} \tind(F, W). 
\]
For example, if $H = K_2^{\bullet \bullet}$, then $N_{c, H^{\bullet \bullet}}$ is $C_c$, and $\bN_{c, H^{\bullet \bullet}}$ is all the graphs on $c$ vertices containing $C_c$ as a subgraph. Thus $\tind(\bN_{c, K_2^{\bullet \bullet}},W) = t(C_c, W)$. 

The first part of the proof is almost the same as in Subsection \ref{sec:profile}. 
We introduce the following definition to understand $\tind(\bN_{c,H^{\bullet \bullet}}, W)$.
\begin{definition}\label{def:MWq2}
Given $W \in \mathcal{W}$, consider a symmetric measurable function $M_{W,H^{\bullet \bullet}}: [0,1]^2 \to \RR$ such that $M_{W,H^{\bullet \bullet}}(x,y) = t_{\text{ind}, x,y}(H^{\bullet \bullet}, W)$.
\end{definition}
This definition is a generalization of Definition \ref{def:MWq}. Indeed, when $H^{\bullet \bullet} = K_q^{\bullet \bullet}$, $M_{W,H^{\bullet \bullet}}$ is the same as $M_{W,q}$ from Definition \ref{def:MWq}. 
The following lemma is immediate.
\begin{lemma}\label{lem:powersum2}
Let $W \in \mathcal{W}$. Then $\tind(\bN_{c, H^{\bullet \bullet}}, W)=\sum_{i} \lambda_i^c$ where $\lambda_1, \lambda_2, \ldots$ are the eigenvalues of $M_{W,H^{\bullet \bullet}}$. 
\end{lemma}

By exactly the same proof as in Lemma \ref{lem:Dq}, if we let $$\mathcal{D}_{H^{\bullet \bullet},q}:=\textup{cl}\left(\left \{ \left(\frac{\tind(\bN_{8,H^{\bullet \bullet}},W)}{\tind(\bN_{4,H^{\bullet \bullet}},W)^2}, \frac{\tind(\bN_{12,H^{\bullet \bullet}},W)}{\tind(\bN_{4,H^{\bullet \bullet}},W)^3}\right) \mid W \in \mathcal{W} \textup{ and }\tind(\bN_{4,H^{\bullet \bullet}},W)\neq 0 \right\}\right),$$  
then $\mathcal{D}_{H^{\bullet \bullet},q}\subseteq [0,1]^2$.  In addition, for any $(x,y)\in \mathcal{D}_{H^{\bullet \bullet},q}$, we have that $y\geq x^2$.  Furthermore, by Lemma \ref{lem:powersum2}, Lemma \ref{lem:convex} holds here too. Recall from Definition \ref{def:R} that $\mathcal{R}$ is the convex hull of the points $\{(\frac{1}{n}, \frac{1}{n^2}) \mid n \in \mathbb{N}\}$ to obtain the following lemma.

\begin{lemma}\label{lem:convex2}
We have that $\mathcal{D}_{H^{\bullet \bullet},q}\subseteq \mathcal{R}$.
\end{lemma}

The second part of the proof serves the purpose of Subsection \ref{subsec:k}. 
Given $k$  different $H^{\bullet \bullet}$,
we now proceed to show that we can obtain $k$ independent copies of the convex hull. 
In particular, we prove the following result, which serves the purpose of Proposition \ref{prop:specialgraph}. This proof is simpler than the proof of  Proposition \ref{prop:specialgraph}, which is the main result in Subsection \ref{subsec:k}. 

\begin{proposition}\label{prop:specialgraph2}
Let $H_1^{\bullet \bullet}, \dots, H_\ell^{\bullet \bullet}$ be $\ell$ different graphs in $\mathcal{H}$ with no induced homomorphism from one to another. Furthermore, assume that each $H_i$ has a vertex cut set of size at least three and contains a triangle, and that there is no induced homomorphism  $V(H_i)\to V(H_i)$ which  does not fix the set of the two roots. 
Let \[ x_i(W) =\frac{\tind(\bN_{8,H_i^{\bullet \bullet}},W)}{\tind(\bN_{4,H_i^{\bullet \bullet}},W)^2}, \ \  y_i(W) = \frac{\tind(\bN_{12,H_i^{\bullet \bullet}},W)}{\tind(\bN_{4,H_i^{\bullet \bullet}},W)^3}, \ \ \text{and} \] 

\begin{align*}\mathcal{D}'_{\leq \ell}:=\textup{cl}\left(\left\{ \left( x_1(W), y_1(W), x_2(W), y_2(W), \dots, x_\ell(W), y_\ell(W)
\right): 
W \in \mathcal{W},   t(N_{4,H_i^{\bullet \bullet}},W)\neq 0 \ \forall 1 \leq i \leq \ell \right\}\right).
 \end{align*} 
Then for any set of positive integers $r_1, \ldots, r_\ell$, $(\frac{1}{r_1}, \frac{1}{r_1^2},\frac{1}{r_2}, \frac{1}{r_2^2}, \ldots, \frac{1}{r_\ell}, \frac{1}{r_\ell^2})\in \mathcal{D}'_{\leq \ell}$. 
\end{proposition}
\begin{proof}
For a positive integer $k$, let $A(k)$ be the following graph constructed by Alon \cite{Alon}:
$A(k)$ is a triangle-free graph with $n=\Theta(2^{3k})$ vertices that is $d = \Theta(2^{2k})$-regular and for which the second largest eigenvalue (in absolute value) is $\Theta(2^k)$. 
For each $1 \leq i \leq \ell$, let $G_i(k)$ be the graph obtained by replacing each edge of $A(k)$ by a copy of $H_i^{\bullet\bullet}$, identifying the two roots to the end vertices of the edge. For simplicity, write $G_i(k), A(k)$ as $G_i, A$ respectively. 

Let $G$ be the disjoint union of $r_i$ copies of $G_i$ for  $1 \leq i \leq \ell$. Let $N$ be the number of vertices in $G$. 
We now study $\tind(\bN_{c, H_i^{\bullet\bullet}}, G)$. Fix $i$ so that, by definition, we are only looking for induced copies of $H_i$. Note that since $H_i$ contains a triangle and $A$ is triangle-free, there is no copy of $H_i$ in $A$. Since $H_i$ is connected and has a vertex cut set of size at least three, each induced copy of $H_i$ in $G$ has to come from one single copy of $H_j$ for some $j$ which is glued to some original edge in $A$. Since there is no induced copy of $H_i$ in $H_j$ for $j\neq i$,  the only induced copies of $H_i$ has to come from $G_i$ and not  $G_j$ for $j \neq i$.  Furthermore, since all induced homomorphisms from $H_i$ to itself fix the set of roots, only the original edges in $A \subset G_i$ could yield an induced copy of $H_i^{\bullet \bullet}$ while identifying the end vertices of the edge to the roots (exactly one copy). Thus, by definition, $M_{G, H_i^{\bullet \bullet}}$ is always zero unless $x,y \in [0,1]^2$ correspond to two end vertices of an edge in $A \subset G_i$, in which case $M_{G, H_i^{\bullet \bullet}}(x,y)=\frac{1}{N^{|V(H_i)|-2}}$. Let $\lambda_1 \geq \lambda_2 \dots$ be the eigenvalues of $M_{G, H_i^{\bullet \bullet}}$. Note that they are precisely the eigenvalues of $A$, each duplicated $r_i$ times while scaled by the same non-zero scalar. Thus if $|\mu_1| > |\mu_2| \geq  \dots$ are the eigenvalues of the adjacency matrix of graph $A$, then by Lemma \ref{lem:powersum2} and the properties of $A(k)$, we have \[x_i(W_{G}) = \frac{r_i\sum_j \mu_j^8}{(r_i\sum_j  \mu_j^4)^2}
=  \frac{1}{r_i}\frac{(d^8 +  n O(d^{1/2})^8)}{((d^4 +  n O(d^{1/2})^4))^2} =  \frac{1}{r_i}\frac{d^8 +  \Theta(d^{1.5}) O(d^{1/2})^8}{(d^4 + \Theta(d^{1.5}) O(d^{1/2})^4)^2} \to  \frac{1}{r_i},
\]
where we use $d = \Theta(2^{2k})$ and $k \to \infty$. Similarly, we have $y_i(W_G) \to 1/r_i^2$ as $k \to \infty$
\end{proof}

There are many ways to construct graphs $H_i$'s satisfying the assumptions in Proposition \ref{prop:specialgraph2}. For example, one can build each $H_i$ from a Cayley graph, fixing two vertices as roots, and then changing some of its non-neighbors to neighbors. As long as the degrees of these two roots are larger than the degrees of the other vertices, there will be no induced homomorphism from $H_i$ to itself that does not fix the set of roots. Similarly, if $H_i$ and $H_j$ have different degree sequences while $|V(H_i)| = |V(H_j)|$,  then there is no induced homomorphism from one to another. 

We can now complete the proof of Theorem \ref{thm:main} using the same argument as in Subsection \ref{subsec:complete}. Given a polynomial $p$ in variables $x_1, \ldots, x_\ell$, let $\bar p$ be defined in the same way as in Lemma \ref{lem:pbar}. First note that by the definition of $\tind$, we have that for each graph $H$, there is a quantum graph $g$ such that for any $W \in \mathcal{W}$, $\tind(H, W) = t(g, W)$. Therefore there is a quantum graph $f(p)$ such that
\begin{align}
    t(f(p), W):= \bar{p}\left(\frac{\tind(\bN_{8,H_1^{\bullet \bullet}},W)}{\tind(\bN_{4,H_1^{\bullet \bullet}},W)^2}, \ldots, \frac{\tind(\bN_{8,H_\ell^{\bullet \bullet}},W)}{\tind(\bN_{4,H_\ell^{\bullet \bullet}},W)^2}, \frac{\tind(\bN_{12,H_1^{\bullet \bullet}},W)}{\tind(\bN_{4,H_1^{\bullet \bullet}},W)^3}, \ldots, \frac{\tind(\bN_{12,H_\ell^{\bullet \bullet}},W)}{\tind(\bN_{4,H_\ell^{\bullet \bullet}},W)^3}\right)  \nonumber
    \\ \prod_{i=1}^\ell \tind(\bN_{4,H_\ell^{\bullet \bullet}},W)^{3\textup{deg}(p)}. \label{eq:quantum2}
\end{align}
Then the rest of the proof follows from the same argument as in Subsection \ref{subsec:complete}. 

\end{document}